\documentclass[12pt]{article}

\usepackage{amsmath}
\usepackage{amsfonts}
\usepackage{amssymb}

\usepackage{amsthm}

\newtheorem{thm}{Theorem}

\title{Congruence Lattices of Certain Finite Algebras with Three Commutative Binary Operations}
\author{Brian T. Chan}
\date{July 27 - August 19, 2014}

\begin{document}

\maketitle

\begin{abstract}

A partial algebra construction of Gr\"atzer and Schmidt from $\cite{CLA}$ is adapted to provide an alternative proof to a well-known fact that every finite distributive lattice is representable, seen as special case of the Finite Lattice Representation Problem. \newline

The construction of this proof brings together Birkhoff's representation theorem for finite distributive lattices, an emphasis on boolean lattices when representing finite lattices, and a perspective based on inequalities of partially ordered sets. It may be possible to generalize the techniques used in this approach. \newline

Other than the aforementioned representation theorem only elementary tools are used for the two theorems of this note. In particular there is no reliance on group theoretical concepts or techniques \cite{ISLG}, or on well-known methods, used to show certain finite lattices to be representable \cite{CLFA}, such as the closure method.

\footnotetext {The theorem was formulated while holding a summer 2014 NSERC USRA, supervised by Claude Laflamme and Robert Woodrow.}

\end{abstract}

\section{Introduction}

The Finite Lattice Representation Problem asks if, given a finite lattice $L$, is $L$ representable (see \cite{CLFA}). That is, is there a finite algebra $A$ where $L$ is isomorphic to the congruence lattice of $A$. So far the problem is still open. See \cite{ISLG}, \cite{CLFA} and \cite{CLFUA} for explorations into the general case.\footnote{In \cite{ISLG} (a paper by P\'alfy and Pud\'lak), it is shown that every finite lattice is representable exactly when every finite lattice is isomorphic to an interval in the subgroup lattice of a finite group. And in \cite{CLFA} (DeMeo's PhD thesis) every lattice (up to isomorphism) with at most seven elements was, with at most one possible exception, shown to be representable.} To the best of our knowledge, this problem was first stated in \cite{CLA} by Gr\"atzer and Schmidt in 1963 as an open problem. \newline

The special case of the Finite Lattice Representation Problem investigated in this note asks if every finite distributive lattice can be represented by some finite algebra. In the 1940's, Robert P. Dilworth proved a stronger variant of this special case. Called Dilworth's Theorem $\cite{CLL}$, it states that every finite distributive lattice can be represented as the congruence lattice of some finite lattice.\footnote{He had in mind a conjecture $\cite{CLL}$, called the congruence lattice problem. The problem investigates the limitations of what congruence lattices of lattices could be by asking whether every algebraic distributive lattice is isomorphic to the congruence lattice of some lattice. In 2007 Friedrich Wehrung \cite{SDCLP} showed a counterexample by constructing an uncountably infinite distributive algebraic lattice that is not a congruence lattice of any lattice.} There has been a lot of research into congruence lattices of finite lattices (see this book by Gr\"atzer \cite{CFL}) and on finite algebras whose congruence lattices are distributive, modular, upper semi-modular, or lower semi-modular (see Berman's PhD thesis \cite{CLFUA}). \newline

The algebras described in the two theorems of this note are not lattices, each consists of a join semilattice along with two other commutative binary operations, and so provide an alternative different from various approaches, see \cite{CFL} for techniques used to prove stronger versions of Dilworth's Theorem, considered when looking at congruence lattices of finite lattices. \newline

Two theorems will be shown. The first describes representations of a finite distributive lattice $D$ using two algebras $E$ and $E^\prime$ where $D \cong$ Con $E \cong$ Con $E^\prime$; $E$ has $L$ as its set of elements and $E^\prime$ has a finite boolean lattice with $n$ atoms (with $n$ being the number of join irreducible elements of $D$) as its set of elements. Congruences on these algebras will be shown to be generalizations of congruences on finite boolean lattices. \newline

The second adds to the first by intertwining this representation with the partial orders induced by the semilattice operations of $E$ and $E^\prime$. It appears to indicate how much the representing algebras are like the distributive lattice being represented. \newline

In this note the set theoretic symbols $\mathcal{P}$, $\subset$, $\subseteq$, $\Delta$, $\cup$, $\cap$, and $\backslash$ will be used. Mostly following \cite{ILO}, an important structure, the lattice, will be introduced below along with some related concepts. \newline 

Let $P$ be a set, then a binary relation $\leq$ $\subseteq P \times P$ is a partial order on $P$ if and only if $\leq$ is reflexive, transitive, and antisymmetric. Then $(P, \leq)$, more briefly $P$, is a partially ordered set partially ordered by $\leq$.\footnote{Depending on context $P$ may stand for a set or for $(P, \leq)$.} We will write $a \leq b$ to mean that $(a,b) \in$ $\leq$ and $a < b$ to mean that $a \leq b$ and $a \neq b$. The following is called an interval. If $P$ is a partially ordered set, $a,b \in P$, and $a \leq b$ then define $[a,b] = \{x \in P : a \leq x \leq b \}$. \newline

If, where $P$ is a partially ordered set, $P$ has an element $x \in P$ such that $x \leq y$ for all $y \in P$ then say that $x$ is the bottom of $P$, and denote this bottom by $0$. A top element is defined dually with respect to the partial order of $P$ and is denoted $1$ when it exists. Moreover, if $x,y \in P$ say that $y$ covers $x$ iff $x < y$ and there is no $z \in P$ where $x < z < y$, say that $x$ and $y$ are comparable iff $x \leq y$ or $y \leq x$, and write $x \parallel y$ if $x$ and $y$ are incomparable.\footnote{$x$ and $y$ are comparable iff $x \leq y$ or $y \leq x$} \newline

A lattice is a partially ordered set $P$ that has two binary operations, the meet ($\wedge : P \times P \rightarrow P$) and the join ($\vee : P \times P \rightarrow P$), where $a \wedge b$ is the greatest lower bound of $\{a,b\}$ and $a \vee b$ is the least upper bound of $\{a, b\}$.\footnote{As binary operations, both the meet and the join are commutative, idempotent and associative.} With $L$ being a lattice, a subset $S \subseteq L$ is a sublattice of $L$ exactly when it is closed under the meet and join operations of $L$. \newline

Let $L$ be a lattice and $S \subseteq L$ a finite subset (so that $S = \{s_1, \dots, s_n\}$), then $\bigvee S$ stands for the least upper bound of all the elements of $S$ (which exists). The element $\bigvee S$ can also be written $\bigvee_{k=1}^n s_k$.\footnote{Since joins are associative it can be deduced that $\bigvee S = s_1 \vee \dots \vee s_n$.} A similar meaning (dual with respect to the partial order of $L$) applies to the symbol $\bigwedge S$.  \newline

An equivalent algebraic definition of a lattice $L$ is as follows. A lattice is the algebra $\langle L ; \wedge, \vee \rangle$, being a set $L$ equipped with two commutative, idempotent, and associative binary operations $\wedge : L \times L \rightarrow L$ and $\wedge : L \times L \rightarrow L$ (called the meet and join respectively) such that the absorption laws hold: for all $a, b \in L$, $a \vee (a \wedge b) = a \wedge (a \vee b) = a$.\footnote{Depending on context, $L$ may stand for a set, $(L, \leq)$, or $\langle L ; \wedge, \vee \rangle$. The partial order of $L$ that would make this meet and join conform to the previous definition can be obtained by having, $a \leq_1 b$ iff $a = a \wedge b$, or by having $a \leq_2 b$ iff $b = a \vee b$. Note that $\leq_1$ $=$ $\leq_2$.} See \cite{ILO} for more on this equivalence \newline

Call a lattice that has both a top and a bottom element bounded, and call a lattice complemented if it is bounded and every element has a complement; that is for all $x \in L$ with $L$ being a bounded lattice having top $1$ and bottom $0$, a complement is an element $y \in L$ where $x \wedge y = 0$ and $x \vee y = 1$. Furthermore, a uniquely complemented lattice is defined to be a bounded lattice in which every element has a unique complement. \newline

Considering the lattice $L$, an element $x \in L$ is join irreducible if and only if $x = a \vee b$ in $L$ implies $x = a$ or $x = b$, and an atom if and only if $L$ has a bottom $0$ that is covered by $x$.  In particular, an atom is always join-irreducible; though the converse is not true. Denote the partial order of join irreducible elements partially ordered by the partial order of $L$ by $J(L)$.\footnote{Depending on context, $J(L)$ will be interpreted as a set or as a partial order.} \newline

Let $L$ be a lattice, then it is called distributive exactly when it satisfies the following distributive laws: for all $a,b,c \in L$, $a \vee (b \wedge c) = (a \vee b) \wedge (a \vee c)$ and $a \wedge (b \vee c) = (a \wedge b) \vee (a \wedge c)$. A lattice is called a boolean lattice precisely when it is distributive and complemented.\footnote{In a boolean lattice, the distributive laws can be used to show that every element has a unique complement, so a boolean lattice is uniquely complemented.} A very important property is that in any boolean lattice, a join irreducible element is an atom. \newline

Going back a bit, a semilattice is a partially ordered set that is closed under meets or under joins.\footnote{Like with lattices, the binary operation (being the meet or the join) of a semilattice is commutative, idempotent and associative. Furthermore, the symbols $\bigvee T$ and $\bigwedge T$ for finite subsets $T$ of a semilattice $S$ are defined as they were for lattices. } If joins/meets are to be emphasized call the structure a join/meet semilattice. A lattice can be interpreted as being a join and a meet semilattice being fused together. \newline

Like with a lattice, a semilattice can be defined in two equivalent ways. Algebraically, it can be defined as being a set with a binary operation $* : S \times S \rightarrow S$ that is commutative, idempotent and associative.\footnote{From an algebraic definition of a semilattice, partial orders that would have the semilattice operation conform with the previous definition would be one of the following. When $* = \wedge$ have $a \leq b$ iff $a = a * b$, and when $* = \vee$ have $a \leq b$ iff $b = a * b$.} \newline

Going in another direction, let $A$ be an algebra. Then a congruence $\theta$ on $A$ is a partition of $A$ where for any $n$-ary operation $f$ of A, $x_i \equiv y_i$  $(\theta)$ for $i = 1, \dots, n$ implies that $f(x_1, \dots x_n) \equiv f(y_1, \dots y_n)$ $(\theta)$; \footnote{Write $x \equiv y$ $(\theta)$ to mean that $x$ and $y$ are in the same cell of $\theta$.} and the congruence lattice of $A$, Con $A$ with partial order $\leq$, is the partially ordered set of congruences of $A$ where $\theta \leq \phi$ if and only if $\theta$ is a refinement of $\phi$. \newline

Now, let $P$ be a partially ordered set. Then write $\mathcal{O}(P)$ to denote the lattice of subsets of $P$ closed downward under the partial order of $P$ where the meet and join are, for all $A,B \in \mathcal{O}(P)$, $A \wedge B = A \cap B$ and $A \vee B = A \cup B$ respectively.\footnote{Depending on context, $\mathcal{O}(P)$ will be interpreted as a set or as a lattice.} \newline

With $P$ as above let $S \subseteq P$. Then define $\downarrow_P S = \bigcap \{ X \in \mathcal{O}(P) : S \subseteq X \}$ and let $S^M$ denote the set of maximal, with respect to the partial order of $P$, elements of $S$.\footnote{In particular, $\downarrow_P \varnothing = \varnothing$. To remove possible ambiguity, let $P$ be a partially ordered set and $S \subseteq P$. Then $\downarrow_P S$ will also be denoted by $\downarrow_P (S)$.} \newline

Birkhoff's representation theorem for finite distributive lattices shows that if $D$ is a finite distributive lattice then $\mathcal{O}(J(D)) \cong D$, the identification being $X \mapsto \bigvee X$ for all $X \in \mathcal{O}(J(D))$. In particular, the operations of $D$ are identified with intersections and unions on $\mathcal{O}(J(D))$. This representation theorem indicates that finite distributive lattices are sublattices of finite power sets; in particular, every finite boolean lattice $B$ is isomorphic to the powerset of the (finite) set of atoms of $B$.

\section{The Finite Algebras}

Given a finite distributive lattice $D$ with $n$ join irreducible elements, we construct algebras $E$ and $E^\prime$ such that $D \cong$ Con $E \cong$ Con $E^\prime$. And it will be a generalization of the following: It is well-known that if $B$ is a finite boolean lattice then Con $B \cong B$ and that every congruence on $B$ is uniquely determined by the cell that contains its bottom element $0$. \newline

Looking at Birkhoff's representation, let $D$ be a finite distributive lattice. Then $\langle \mathcal{O}(J(D)) ; \cap, \cup \rangle$ is a distributive lattice that is isomorphic to $D$ and, with $|J(D)| = n$, $\langle \mathcal{P}(J(D)) ; \cap, \cup \rangle$ is a boolean lattice isomorphic to a finite boolean lattice with $n$ atoms since $\mathcal{P}(J(D)) = \mathcal{O}(J(B_n))$. The algebra $E$ is built from  $\langle \mathcal{O}(J(D)) ; \cap, \cup \rangle$ while $E^\prime$ is built from $\langle \mathcal{P}(J(D)) ; \cap, \cup \rangle$. \newline

Below are two theorems. The first establishes how the finite distributive lattices can be represented, identifies the congruences of these algebras, and reveals some algebraic properties of the representing structures. The second adds to the first by describing how inequalities on these algebras can be used to learn more about these algebras.

\begin{thm}
Let $\langle D ; \wedge , \vee \rangle$ be a finite distributive lattice with partial order $\leq$, and $\langle B_n ; \curlywedge, \curlyvee \rangle$ be a finite boolean lattice with $n = |J(D)|$ atoms and partial order $\preceq$. \newline

Then there are algebras $E = \langle D ; \#, \$, \vee \rangle$ and $E^\prime = \langle B_n ; (\#), (\$), \curlyvee \rangle$ where $D \cong$ Con $E \cong$ Con $E^\prime$. \newline

With $(\#)$, $(\$)$, and $\curlyvee$ corresponding to $\#$, $\$$ and $\vee$ respectively, there is an onto homomorphism from $E^\prime$ onto $E$. And looking at the operations all are binary and commutative, both $\$$ and $\#$ are idempotent, $(\#)$ is associative, and for all $a \in B_n$: $(a$ $(\#)$ $a)$ $(\#)$ $a = a$ $(\#)$ $a = a$ $(\$)$ $a = (a$ $(\$)$ $a)$ $(\$)$ $a$. \newline

Furthermore, the congruences of $E$ are the partitions $\theta_a$ where for all $a \in L$: $\theta_a = \{ [0, a] \} \cup \{ x \vee [0, a] : \exists x_1, \dots x_n \in J(D) \backslash \downarrow_{J(D)} \{a\} \text{ } (x = \bigvee_{i=1}^n x_i) \}$ and the congruences of $E^\prime$ are the partitions $\theta_A^\prime$ where for all $A \subseteq J(B_n)$: $\theta_A = \{ [0, \curlyvee A] \} \cup \{ x \curlyvee [0, \curlyvee A] : \exists S \subseteq J(B_n) \backslash A \text{ } (x = \curlyvee S) \}$ 

\end{thm}

\begin{proof}

Below, $\mathcal{O}(J(D))$ will be used instead of $D$ and $\mathcal{P}(J(D))$ will be used instead of $B_n$. Throughout, set $\downarrow$ $=$ $\downarrow_{J(D)}$ and evaluate $S^M$ with respect to the order of $J(L)$. Now define $F = \langle \mathcal{O}(J(D)) ; \#, \$, \cup \rangle$ and $F^\prime = \langle \mathcal{P}(J(D)) ; (\#), (\$), \cup \rangle$ both ordered by $\subseteq$; they can be used in place of $E$ and $E^\prime$. \newline
 \newline
For all $X, Y \in \mathcal{P}(J(D))$ define \newline
 \newline
$X$ $(\#)$ $Y =$ $X^M$ $\cap$ $Y^M$ \newline
 \newline
$X$ $(\$)$ $Y =$ $(X^M \cap$ $\downarrow Y) \cup ($ $\downarrow X \cap Y^M)$ \newline 

It can be observed that $\downarrow X$ $\cup$ $\downarrow Y =$ $\downarrow ($ $(\downarrow X)^M$ $\cup$ $(\downarrow Y)^M) =$ $\downarrow (X \cup Y)$. Now for all $A, B \in \mathcal{O}(J(L))$ define $A$ $\#$ $B =$ $\downarrow (A$ $(\#)$ $B)$ and $A$ $\$$ $B =$ $\downarrow (A$ $(\$)$ $B)$. An immediate consequence is that $A$ $\$$ $B =$ $\downarrow((A^M \cap$ $B) \cup ($ $A \cap B^M))$. \newline 
 
Identifying $D$ with $\mathcal{O}(J(D))$ and $B_n$ with $\mathcal{P}(J(D))$, after using implicitly a fixed but arbitrary bijection $m : J(B_n) \rightarrow J(L)$, define $f : E^\prime \rightarrow E$ by $f(X) =$ $\downarrow X$ for all $X \in \mathcal{P}(J(D))$. \newline 
 
This can be seen to be an onto homomorphism as specified in the theorem. It is evident that both $(\$)$ and $(\#)$ are commutative, that for all $A \in \mathcal{P}(J(L))$ $(A$ $(\#)$ $A)$ $(\#)$ $A = A$ $(\#)$ $A = A$ $(\$)$ $A = (A$ $(\$)$ $A)$ $(\$)$ $A$, and that $\#$ and $\$$ are both idempotent. As $(X$ $(\#)$ $Y)$ $(\#)$ $Z = X^M \cap Y^M \cap Z^M$ and similarly for $X$ $(\#)$ $(Y$ $(\#)$ $Z)$, $(\#)$ is also associative. \newline
 \newline
Next we argue the representation $\mathcal{O}(J(D)) \cong$ Con $F \cong$ Con $F^\prime$. \newline 

Let $\theta \in$ Con $F$. Assume that $X \equiv Y$ $(\theta)$ and $X^M \cap Y^M = \varnothing$. Then $X = $ $\downarrow (X^M \cap X^M) \equiv$ $\downarrow (X^M \cap Y^M) = \varnothing$ $(\theta)$ and similarly for $Y$. Next it will be shown that $X \equiv Y$ $(\theta)$ iff $\downarrow (X^M \backslash Y^M) \equiv$ $\downarrow(Y^M \backslash X^M)$ $(\theta)$. \newline

Assume that $X \equiv Y$ $(\theta)$, then $\downarrow (X^M \backslash Y^M) = X$ $\#$ $\downarrow((X^M \cup Y^M) \backslash (X^M \cap Y^M)) \equiv Y$ $\#$ $\downarrow((X^M \cup Y^M) \backslash (X^M \cap Y^M)) =$ $\downarrow(Y^M \backslash X^M)$ $(\theta)$. Conversely, $\downarrow(X^M \backslash Y^M) \equiv$ $\downarrow (Y^M \backslash X^M)$ $(\theta)$ implies $X =$ $\downarrow(X^M \backslash Y^M) \cup (X \# Y) \equiv$ $\downarrow(Y^M \backslash X^M) \cup (X \# Y) = Y$ $(\theta)$. \newline

When $\theta \in$ Con $F^\prime$ argue like before (starting with the case $X^M \cap Y^M = \varnothing$) except restrict $X$ and $Y$ in the above by having $X, Y \in \{S^M : S \in \mathcal{P}(J(L)) \}$. \newline

Consider $I(\theta)$, the block of $\theta$ containing $\varnothing$. If $\varnothing \neq A \in I(\theta)$ and $a \in$ $A$ then observe that $\downarrow \{a\} = A$ $\$$ $\downarrow \{a\} \equiv \varnothing$ $\$$ $\downarrow \{a\} = \varnothing$ $(\theta)$. With $F^\prime$ replace $\downarrow \{a\}$ with $\{a\}$. Now have $I = \cup I(\theta)$. Then $I(\theta) = \mathcal{O}(I)$ when considering $F$ and $I(\theta) = \mathcal{P}(I)$ when considering $F^\prime$. The other cells then take the form $\{X \cup Y : Y \in I(\theta) \}$ for $X \subseteq J(L)$ and $X \nsubseteq I$ (replace $X$ with $\downarrow X$ when considering $F$). \newline

In either case, using the binary operation $\cup$ with the above shows that $\theta$ is uniquely determined by $I(\theta)$ and that $I(\theta)$ is identified with an element of $\mathcal{O}(J(L))$. \newline

Now let $I \in \mathcal{O}(J(D))$ and consider the partition $\phi$ on $\mathcal{O}(J(L))$ where $X \equiv Y$ $(\phi)$ if and only if $X^M \Delta Y^M \subseteq I$. Let $A,X,Y \in \mathcal{O}(J(L))$ and $X \equiv Y$ $(\phi)$. Then $(A \cup X)^M \Delta (A \cup Y)^M \subseteq (A^M \cup X^M) \Delta (A^M \cup Y^M) \subseteq I$, $(A$ $\#$ $X)^M \Delta (A$ $\#$ $Y)^M = (A^M \cap X^M) \Delta (A^M \cap Y^M) \subseteq I$, and $(A$ $\$$ $X)^M \Delta (A$ $\$$ $Y)^M \subseteq$ $\downarrow X$ $\Delta$ $\downarrow Y \subseteq$ $\downarrow(X^M \Delta Y^M) \subseteq I$. So $\phi \in$ Con $F$. With $F^\prime$ use the same definition when making a partition $\phi$ of $\mathcal{P}(J(L))$ from a given $I \in \mathcal{O}(J(L))$, namely $X^M \Delta Y^M \subseteq I$.
\end{proof}

The smaller algebra, $E$, is a homomorphic image of $E^\prime$ just like how $\langle L ; \vee \rangle$ is a homomorphic image of $\langle B_n; \curlyvee \rangle$ (with $n = |J(D)|$); we show how certain inequalities involving $\leq$ and $\preceq$ can be used to gauge, with $D = E \cong E^\prime$ being at one extreme, how much $E$ and $E^\prime$ are like $D$.
\begin{thm}
Let $\langle D ; \wedge , \vee \rangle$ be a finite distributive lattice with partial order $\leq$, and $\langle B_n ; \curlywedge, \curlyvee \rangle$ be a finite boolean lattice with $n = |J(D)|$ atoms and partial order $\preceq$. Then there are algebras $E$ and $E^\prime$, both being as specified in the preceding theorem, such that when viewed as pairs $(E, \leq)$ and $(E^\prime, \preceq)$ the following can be said. Some governing inequalities are as follows, for all $a, b \in B_n$ and $c, d \in L$:
\begin{center}
$a$ $(\#)$ $b \preceq a$ $(\$)$ $b \preceq a \curlyvee b$ and $a$ $(\#)$ $b \preceq a \curlywedge b \preceq a \curlyvee b$
\end{center}
\begin{center}
$c$ $\#$ $d \leq c$ $\$$ $d \leq c \wedge d \leq c \vee d$
\end{center}
Furthermore, the following four properties are equivalent:

\begin{enumerate}
\item for all $a,b \in B_n$: $a \curlywedge b \preceq a$ $(\$)$ $b$
\item $\$$ is associative
\item $\$ = \wedge$
\item it is impossible to find three elements $x,y,z \in J(L)$ where $y \parallel z$ and $x \leq y \wedge z$
\end{enumerate}
And the following four properties are equivalent:

\begin{enumerate}
\item for all $a,b \in B_n$: $a$ $(\$)$ $b \preceq a \curlywedge b$
\item $(\#)$ and $(\$)$ are idempotent
\item $\$ = \#$
\item $D$ is a boolean lattice
\end{enumerate}
\end{thm}

\begin{proof} Like before, write $\downarrow$ to denote $\downarrow_{J(D)}$. The operations $(\#)$, $(\$)$, $\#$ and $\$$ are as defined in the proof of the preceding theorem. \newline

To show the governing inequalities, the following can be said. From the setup we can see that $X (\#) Y \subseteq X (\$) Y \subseteq X \cup Y$, $X (\#) Y \subseteq X \cap Y \subseteq X \cup Y$, and $A^M \cap B^M \subseteq (A^M \cap B) \cup (A \cap B^M) \subseteq A \cap B$. In particular we obtain, $(A$ $\#$ $B)^M \subseteq (A$ $\$$ $B)^M \subseteq A \cap B$ and $\bigvee (A$ $\#$ $B) \leq \bigvee (A$ $\$$ $B) \leq \bigvee (A \cap B)$. \newline
. \newline
Looking at the first inequality between $(\$)$ and $\curlywedge$ : \newline

Say that $J(D)$ is forest-like if it is impossible to find elements $x,y,z \in J(D)$ where $y \parallel z$ and $x \leq y \wedge z$. If $J(D)$ were forest-like then for all $X,Y \in \mathcal{P}(J(D))$ $X \cap Y \subseteq X$ $(\$)$ $Y$ for if $x \in (X \cap Y) \backslash (X$ $(\$)$ $Y)$ then $x \notin X^M \cup Y^M$ implying that $J(D)$ is not forest-like. If $J(D)$ were not forest-like let $x,y,z \in J(D)$ where $y \parallel z$ and $x \leq y \wedge z$. Set $X = \{x,y\}$ and $Y = \{x,z\}$, then $X \cap Y = \{x\} \neq \varnothing = X$ $(\$)$ $Y$. \newline

Assume now that $J(D)$ is forest-like and let $A,B,C \in \mathcal{O}(J(D))$. Then it is enough to suppose that $(A$ $\$$ $(B$ $\$$ $C))^M \neq$ $(A^M \cap B \cap C) \cup (A \cap B^M \cap C) \cup (A \cap B \cap C^M)$. Firstly, $B^M \cap (A$ $\$$ $(B$ $\$$ $C))^M \subseteq B^M \cap (A \cap B \cap C) = A \cap (B^M \cap C) \subseteq B^M \cap (A$ $\$$ $(B$ $\$$ $C))^M$ and similarly for $A \cap B \cap C^M$. Secondly, $A^M \cap (A$ $\$$ $(B$ $\$$ $C))^M \subseteq A^M \cap (B \cap C)$. So it is impossible that for all $a \in A^M \cap B \cap C$ one can find a $b \in B^M$ and a $c \in C^M$ where $a \leq b \leq c$ or $a \leq c \leq b$. \newline

But then there is a $\{x,y,z\} \subseteq J(D)$ where $y \parallel z$ in $J(D)$, $x < y$ and $x < z$; the latter being impossible, and the former leading to a contradiction. To see that $\$ = \cap$ note that $(A \cap B \cap C)^M = (A^M \cap B \cap C) \cup (A \cap B^M \cap C) \cup (A \cap B \cap C^M) = (A$ $\$$ $B$ $\$$ $C)^M$, then set $B = C$. \newline

Now assume that $J(D)$ is not forest-like, and let $x,y,z \in J(L)$ where $y \parallel z$ and $x \leq y \wedge z$. Set $A = $ $\downarrow \{x\}$, $B =$ $\downarrow \{y\}$, and $C = $ $\downarrow \{z\}$. Then $(A$ $\$$ $B)$ $\$$ $C = A \neq \varnothing = A$ $\$$ $(B$ $\$$ $C)$ and $B$ $\$$ $C = \varnothing \subset A = B \cap C$. \newline
 \newline
Looking at the second inequality between $(\$)$ and $\curlywedge$ : \newline

Firstly, if $D$ is a boolean lattice then $(\#) = (\$) = \curlywedge$ and $\# = \$ = \wedge$. In particular both $(\#)$ and $(\$)$ would be idempotent. So assume that $D$ is not a boolean lattice, then let $x, y \in J(D)$ satisfy $x < y$ : \newline

Set $X = \{x\}$ and $Y = \{y\}$, then $X \cap Y = \varnothing \subset \{x\} = X$ $(\$)$ $Y$. Now set $X = \{x, y\}$, then $X$ $(\#)$ $X = X$ $(\$)$ $X = \{y\} \neq X$. At last, let $A =$ $\downarrow \{x\}$ and $B =$ $\downarrow \{y\}$. Then $A$ $\#$ $B = \varnothing \subset$ $\downarrow \{x\} = A$ $\$$ $B$.

\end{proof}

\end{document}